\documentclass{amsart}
\usepackage{amssymb}
\usepackage{enumerate}
\usepackage{stmaryrd}
\usepackage{url}
\usepackage{tikz-cd}
\usepackage[OT1]{fontenc}

\renewcommand{\leq}{\leqslant}
\renewcommand{\geq}{\geqslant}

\renewcommand{\trianglelefteq}{\trianglelefteqslant}

\newtheorem{theorem}{Theorem}[section]

\newtheorem{convention}[theorem]{Convention}
\newtheorem{corollary}[theorem]{Corollary}
\newtheorem{definition}[theorem]{Definition}
\newtheorem{lemma}[theorem]{Lemma}
\newtheorem{proposition}[theorem]{Proposition}
\newtheorem{hypothesis}[theorem]{Hypothesis}
\newtheorem{fact}[theorem]{Fact}
\newtheorem{remark}[theorem]{Remark}
\newtheorem{notation}[theorem]{Notation}

\newtheorem*{theorem1.2}{Theorem 1.2}
\newtheorem*{theorem1.3}{Theorem 1.3}
\newtheorem*{corollary1.4}{Corollary 1.4}
\newtheorem*{corollary1.5}{Corollary 1.5}

\def\Ind#1#2{#1\setbox0=\hbox{$#1x$}\kern\wd0\hbox to 0pt{\hss$#1\mid$\hss}
\lower.9\ht0\hbox to 0pt{\hss$#1\smile$\hss}\kern\wd0}

\newcommand{\indep}[2]{%
  \mathrel{
    \mathop{
      \vcenter{
        \hbox{%
\oalign{
\noalign{\kern-.3ex}\hfil$\vert$\hfil\cr
              \noalign{\kern-.7ex}
              $\smile$\cr\noalign{\kern-.3ex}
}
}
      }
}^{\!\!\!\!\!#2}_{\!\!\hspace{-0.1em}#1}
  }
}

\newcommand{\displayindep}[2]{%
  \mathrel{
    \mathop{
      \vcenter{
        \hbox{%
\oalign{
\noalign{\kern-.3ex}\hfil$\vert$\hfil\cr
              \noalign{\kern-.7ex}
              $\smile$\cr\noalign{\kern-.3ex}
}
}
      }
}^{\!\!\hspace{-0.1em}#2}_{\!\!\hspace{-0.1em}#1}
  }
}

\newcommand{\displayfindep}[2]{%
  \mathrel{
    \mathop{
      \vcenter{
        \hbox{%
\oalign{
\noalign{\kern-.3ex}\hfil$\vert$\hfil\cr
              \noalign{\kern-.7ex}
              $\smile$\cr\noalign{\kern-.3ex} 
}
}
      }
}^{\!\hspace{-0.14em}#2}_{\!\!\hspace{-0.05em}#1}
  }
}

\newcommand{\rkindep}[1][]{\indep{#1}{\mathrm{rk}}}
\newcommand{\nrkindep}[1][]{\not\indep{#1}{\mathrm{rk}}}

\def\ind{\mathop{\mathpalette\Ind{}}}

\def\notind#1#2{#1\setbox0=\hbox{$#1x$}\kern\wd0
\hbox to 0pt{\mathchardef\nn=12854\hss$#1\nn$\kern1.4\wd0\hss}
\hbox to 0pt{\hss$#1\mid$\hss}\lower.9\ht0 \hbox to 0pt{\hss$#1\smile$\hss}\kern\wd0}

\begin{document}
\begin{abstract} Generalizing the $\omega$-categorical context, we introduce a notion, which we call the Lascar Property, that allows for a fine analysis of the topological isomorphisms between automorphism groups of countable saturated structures satisfying this property. In particular, under these assumptions, we exhibit a Galois correspondence between pointwise stabilizers of finitely generated algebraically closed subsets of $M$ and finitely generated algebraically closed subsets of $M$. We use this to characterize the group of automorphisms of $\mathrm{Aut}(M)$, for $M$ a countable saturated model of $\mathrm{ACF}_0$ or an infinite-dimensional $\mathbb{K}$-vector space with $\mathbb{K}$ countable, generalizing a classical result of Evans $\&$ Lascar (1997), while at the same time subsuming the analysis of Paolini (2024) for $\omega$-categorical structures with weak elimination of imaginaries.
\end{abstract}

\title[A Galois correspondence for automorphism groups]{A Galois correspondence for automorphism groups of structures with the Lascar property}

\thanks{Research of the first-named author was  supported by the project PRIN 2022 ``Models, sets and classifications", prot. 2022TECZJA, and by INdAM Project 2024 (Consolidator grant) ``Groups, Crystals and Classifications''. The authors wish to thank D. Evans for crucial discussions related to this paper and Emil Jeřábek for his answer on StackExchange to a question posed by Federico Pisciotta.}

\author{Gianluca Paolini}
\address{Department of Mathematics ``Giuseppe Peano'', University of Torino, Via Carlo Alberto 10, 10123, Italy.}
\email{gianluca.paolini@unito.it}

\author{Federico Pisciotta}
\address{Department of Mathematics, Computer Science and Physics, University of Udine, Via delle Scienze 206, 33100, Italy.}
\email{pisciotta.federico@spes.uniud.it}

\maketitle


\section{Introduction}

In \cite{Paolini_BLMS}, inspired by a descriptive set theoretic question from \cite{coarse}, we proved that, under the assumption of weak elimination of imaginaries, we have that the open subgroups of a countable $\omega$-categorical structure $M$ admit a combinatorial characterization as ``generalized setwise stabilizers'' of finitely generated Galois algebraically closed subsets of $M$. This led to a partial solution to one of the problems posed in \cite{coarse}, and also, more algebraically, to a concrete description of the group $\mathrm{Aut}_{\mathrm{top}}(\mathrm{Aut}(M))$ of topological automorphisms of $\mathrm{Aut}(M)$, where on $\mathrm{Aut}(M)$ we consider the topology induced by the usual topology of finite information on $\mathrm{Sym}(\omega)$. The point of this paper is to extend the techniques introduced in \cite{Paolini_BLMS} (and later also explored in \cite{coarse_PE}) outside of the context of $\omega$-categorical structures, with connections to classical works of Evans $\&$ Lascar from \cite{Evans_ACF} on outer automorphisms of groups of automorphisms of $\omega$-saturated models of strongly minimal theories.

From now on, all the structures $M$ considered in this paper will be countable and in a countable language, and $G=\mathrm{Aut}(M)$. Also, we will denote by $\mathbf{A}(M)$ the set of algebraically closed sets of $M$ of the form $\mathrm{acl}_M(A)$ for $A$ a finite subset of $M$, where ``algebraic closure'' is the standard model theoretic operator, cf. \ref{algebraicity}. As is well known, for $\omega$-categorical structures $M$, the hypothesis of weak elimination of imaginaries admits many equivalent definitions; the one that will be relevant to our paper, and which was crucial for \cite{Paolini_BLMS}, is the following: for every open subgroup $H$ of $G = \mathrm{Aut}(M)$ there is a (unique) $K \in \mathbf{A}(M)$ such that $G_{(K)} \leqslant H \leqslant G_{\{K\}}$, where $G_{(K)}$ and $G_{\{K\}}$ denote ``pointwise stabilizer'' and ``setwise stabilizer'', respectively. This condition arose naturally also outside of the $\omega$-categorical context, and it appeared elsewhere in the literature, in particular in the works of Poizat \cite[Chapter 16]{poizat} on one side, and Evans and Lascar \cite{Evans_ACF} on the other side. We will see that this condition, together with a saturation assumption, allows for the techniques from \cite{Paolini_BLMS} to be applied outside of the context of $\omega$-categorical structures (while at the same time subsuming the $\omega$-categorical context). This leads to our standing hypothesis:

\begin{hypothesis}\label{hyp_intro}
Let $M$ be a countable saturated structure with the Lascar Property, i.e., for every open $H \leqslant G$, there is $K \in \mathbf{A}(M)$ such that:
$$G_{(K)} \leqslant H \leqslant G_{{K}}.$$
\end{hypothesis}

As we shall see, in the context of (infinite) $\omega$-categorical structures, the assumption of saturation is automatically satisfied, while the Lascar Property corresponds {\em exactly} to weak elimination of imaginaries. This is the context of the aforementioned reference \cite{Paolini_BLMS}. The next theorem shows that the analysis undertaken in \cite{Paolini_BLMS} generalizes smoothly to the general context of countable structures satisfying \ref{hyp_intro}, and so in particular the requirement of local finiteness of the Galois algebraic closure operator (which is satisfied in the $\omega$-categorical context) is not necessary.

 \begin{theorem}\label{main_th} 
    Let $M$ and $N$ be countable saturated structures with the Lascar Property. Then the following holds:
 \begin{enumerate}[(1)]
 \item pointwise stabilizers of finitely generated algebraically closed subsets of $M$ are $G_\delta$-subgroups of $\mathrm{Aut}(M)$, and there is a Galois correspondence between them and the set of finitely generated algebraically closed subsets of $M$;
 \item the structure $\mathcal{E}^{\mathrm{ex}}_M$ associated to $M$ as in Definition \ref{def_expanded_12} is such that: $$\mathrm{Aut}(M) \cong_{\mathrm{top}} \mathrm{Aut}(N) \; \Leftrightarrow \; \mathcal{E}^{\mathrm{ex}}_M \cong \mathcal{E}^{\mathrm{ex}}_N;$$
 and $$\mathrm{Aut}_{\mathrm{top}}(\mathrm{Aut}(M)) \cong \mathrm{Aut}(\mathcal{E}^{\mathrm{ex}}_M).$$
 \end{enumerate}
 \end{theorem}

In many respects, Theorem~\ref{main_th} can be considered the best possible in terms of describing topological isomorphisms. In fact, in many cases of interest it suffices to examine only a single slice, or a finite number of slices, of $\mathcal{E}^{\mathrm{ex}}_M$ (see \ref{main_theorem} for a precise statement), and this already yields a concrete characterization of $\mathrm{Aut}_{\mathrm{top}}(\mathrm{Aut}(M))$. We will see that \ref{main_th} has several explicit applications, but before turning to these, it is natural to ask: when are our assumptions satisfied? We dwell on this question in Section~\ref{section4}, where we prove that the existence of an integer-valued rank function inducing a stationary independence relation which admits weak canonical bases (in the sense of \ref{def_WCB}), together with a Noetherianity condition on $\mathbf{A}(M)$, suffices to guarantee the Lascar Property (cf.~\ref{th_Lascar_baguette}). This subsumes various cases of interest, and in particular several classical structures of finite Morley rank—most notably the theory of infinite-dimensional $\mathbb{K}$-vector spaces (for $\mathbb{K}$ a countable field) and the theory $\mathrm{ACF}_0$ of algebraically closed fields of characteristic~$0$.

We finally arrive at the applications announced for our methods. In \cite{Evans_ACF}, Evans $\&$ Lascar proved that the automorphism group of the countable saturated model of $\mathrm{ACF}_0$ is complete (i.e., the group has a trivial center and every automorphism is inner). We will use our technology to generalize their proof. 

Our main tool in this respect is the following theorem, in which the group $H$ can be viewed as a slight enlargement of $\mathrm{Aut}(M)$, obtained by considering automorphisms of the underlying geometry. For example, when $M$ is a $\mathbb{K}$-vector space, every automorphism of the geometry (that is, of the associated projective space) is induced by a semilinear automorphism. In this case, $H$ is precisely the group of semilinear automorphisms extending $\mathrm{Aut}(M)$, where the elements of $\mathrm{Aut}(M)$ are exactly the linear automorphisms (see \ref{vectspace_corollary}).

\begin{theorem}\label{genevansth} Let $M$ be a countable saturated structure with the Lascar Property and suppose that $(M,\mathrm{acl}_M)$ is a pregeometry.
Suppose further that there is $\mathrm{Aut}(M)\leq H \leq \mathrm{Sym}(\omega)$ such that:
    \begin{enumerate}[(1)]
    \item $\varphi_H:H \rightarrow \mathrm{Aut}(\mathbb{G}_M)$ defined as $\varphi_H(h)(A)=h(A)$ is a well-defined homomorphism;
    \item $\gamma: H \rightarrow \mathrm{Aut}_{\mathrm{top}}(\mathrm{Aut}(M))$ defined as $h \mapsto h(\cdot)h^{-1}$ is well defined;
    \item there exists {a} homomorphism $\hat{\cdot}:\mathrm{Aut}(\mathbb{G}_M) \rightarrow H$ such that $\pi(\gamma(\hat{g}))=g$ where $\pi$ is the canonical map defined in Lemma \ref{canonicalmaplemma}.
    \end{enumerate} 
Then   $$\mathrm{Aut}_{\mathrm{top}}(\mathrm{Aut}(M))\cong \mathrm{ker}(\pi) \rtimes \mathrm{Aut}(\mathbb{G}_M)$$
and the following diagram commutes:
    $$\begin{tikzcd}
    H \arrow[r, "\varphi_H"] \arrow[d, "\gamma"]
    & \mathrm{Aut}(\mathbb{G}_M) \\
    \mathrm{Aut}_{\mathrm{top}}(\mathrm{Aut}(M)). \arrow{ur}{\pi} 
    &
    \end{tikzcd}$$
    \end{theorem}

This allows us to readily infer the aforementioned results of Evans $\&$ Lascar. Notice that it is known that in the case of $\mathrm{ACF}_0$ we also have the so-called small index property (i.e., every subgroup of countable index is open), and so in this case $\mathrm{Aut}_{\mathrm{top}}(\mathrm{Aut}(M)) = \mathrm{Aut}(\mathrm{Aut}(M))$, i.e., every automorphism of $\mathrm{Aut}(M)$ is automatically topological (this is why we can drop the ``$\mathrm{top}$'' below).

\begin{corollary}\label{acf_corollary} Let $M\models \mathrm{ACF}_0$ be countable and saturated, then we have:
$$\mathrm{Aut}(\mathrm{Aut}(M))\cong \mathrm{Aut}(\mathbb{G}_M),$$

and automorphisms of $\mathrm{Aut}(M)$ are \mbox{inner (so $\mathrm{Aut}(M)$, being centerless, is complete).}
\end{corollary}

Finally, we use \ref{genevansth} to give a description of $\mathrm{Aut}(\mathrm{Aut}(M))$ in the context of vector spaces; this last result generalizes the analogous result from \cite{Paolini_BLMS} in the $\omega$-categorical context. To the best of our knowledge, this last application is new, but proofs using different methods might have appeared elsewhere in the literature. Notice that, in this case as well, the small index property holds (so we can drop the ``$\mathrm{top}$'').

\begin{corollary}\label{vectspace_corollary}
Let $V$ be a vector space of dimension $\aleph_0$ over a countable field $\mathbb{K}$. Then:
$$\mathrm{Aut}(\mathrm{Aut}(V)) \cong \mathrm{Aut}(\mathbb{K}^\times) \rtimes \mathrm{Aut}(\mathbb{G}_V).$$
\end{corollary}

Finally, we observe that the scope of applicability of our methods is wider than one might expect at first glance, if one is willing to pass from $M$ to $M^{\mathrm{eq}}$. By this, we mean that although $M$ may fail \ref{hyp_intro}, adding additional sorts from $M^{\mathrm{eq}}$, we might very well make \ref{hyp_intro} true. This is most easily seen in the case of $\omega$-categorical structures (where condition \ref{hyp_intro} corresponds {\em exactly} to weak elimination of imaginaries), but this remark is \mbox{meaningful also outside of the} $\omega$-categorical context.

Concerning the structure of our paper. In Section~\ref{char_stab} we characterize pointwise stabilizers. In Section~\ref{sec_reconstruction} we prove Theorem~\ref{main_th}. In Section~\ref{section4} we explore when the conditions from \ref{hyp_intro} are satisfied. Finally, in Section~\ref{sec_applications} we focus on applications.

\section{Characterizing pointwise stabilizers}\label{char_stab}

This section provides a topological–algebraic characterization of the pointwise stabilizers of finitely generated algebraically closed subsets of a countable saturated structure satisfying the Lascar Property. This characterization is fundamental for the proof of our Main Theorem (i.e., Theorem~\ref{main_th}). Indeed, once the pointwise stabilizers of finitely generated algebraically closed subsets have been identified as a set of topological subgroups, the machinery developed in the proof of Theorem \ref{main_theorem} can be applied.
The idea of analyzing the lattice of finitely generated algebraically closed sets originates in \cite{Paolini_Shelah}, in the $\omega$-categorical context, and was further developed in \cite{Paolini_BLMS}. Several definitions and proofs in this section are adapted from the latter; we include them to keep the present paper self-contained.

\begin{definition}\label{algebraicity} Let $M$ be a structure.
 \begin{enumerate}[(1)]
 \item $a \in M$ is algebraic  over $A \subseteq M$ if $M\models \varphi(a)\land \exists^{=k} x \varphi(x)$ holds for a formula $\varphi(x) \in L(A)$ (the language of $A$) and some positive integer $k$; 
 \item the algebraic closure of $A \subseteq M$, denoted by $\mathrm{acl}_M(A)$, is the set of elements of $M$ which are algebraic over $A$.
\end{enumerate}
\end{definition}
\begin{notation} \label{algclosednotation}Let $M$ be a structure and $G = \mathrm{Aut}(M)$.
 \begin{enumerate}[(1)]
 \item Given a structure $M$ and $A \subseteq M$, and considering $G$ in its natural action on $M$, we denote the pointwise (resp. setwise) stabilizer of $A$ under this action by $G_{(A)}$ (resp. $G_{\{ A \}}$); 
 \item we let $\mathbf{A}(M) = \{ \mathrm{acl}_M(A) : A \subseteq M, |A|<\omega \}$;
 \item we let $\mathcal{PS}(M) = \{ G_{(K)} : K \in \mathbf{A}(M) \} $;
 \item for $K \in \mathbf{A}(M)$ we denote by $\mathrm{Aut}_M(K)$ the set of automorphisms of $K$ which extend to an automorphism of $M$;
 \item we let $\mathbf{EA}(M) = \{ (K, L) : K \in \mathbf{A}(M) \text{ and } L \leq \mathrm{Aut}_M(K) \}$.
    
\end{enumerate}
\end{notation}

\begin{definition}
    We say that $M$ has the \emph{Lascar Property} if for every open $H \leq G$, there is $K \in \mathbf{A}(M)$ such that:
$$G_{(K)} \leq H \leq G_{\{K\}}.$$
Such $K$ is unique by Proposition \ref{unique_K} and we call it the \emph{support} of $H$.
\end{definition}

\begin{hypothesis}\label{hyp} From now on, throughout this section, let $M$ be a countable saturated structure with the Lascar Property.
\end{hypothesis} 

\begin{fact}
    For an $\omega$-saturated structure $M$, the algebraic closure of a finite $A\subseteq M$  coincides with its Galois-algebraic closure, i.e., 
    $a \in M$ belongs to $\mathrm{acl}_M(A)$ if and only if the orbit of $a$ under $G_{(A)}$, denoted as $\mathcal{O}(a/A)$, is finite (see \cite[10.5]{hodges}).
\end{fact}

\begin{proposition}\label{3petals}
    For every $K,S \in \mathbf{A}(M)$, if $S \nsubseteq K$, then there are $g \in G_{(K)}$ and $h \in G_{(S)}$ such that $g^{-1}hg \notin G_{\{S\}}$.
\end{proposition}

\begin{proof} Let $A,B \subseteq M$ be finite and such that $K=\mathrm{acl}_M(A)$ and $S=\mathrm{acl}_M(B)$.
     Let $s \in S \setminus K$.
     We claim that there exists $g \in G_{(K)}$ such that $g(s) \notin S$.
     
     \noindent By saturation, for any finite $K'\subseteq K$ there exists $f \in G_{(K')}$ such that $f(s)\notin S$. Using the compactness theorem, there exists an elementary extension $M \preccurlyeq M'$ and an $f'\in \mathrm{Aut}(M')_{(K)}$ such that $f'(s)\notin S$.
    Note that $K=\mathrm{acl}_M(A)$ is the same in $M$ and in $M'$, and likewise for $S$.
    We may assume $M'$ to be countable and saturated using the usual countable elementary chain argument and the smallness of $\mathrm{Th}(M)$.
    By the uniqueness of saturated models, $M$ is isomorphic to $M'$. More precisely, since $(M,A,B,s)\equiv (M',A,B,s)$, there is an isomorphism $\sigma: M\cong M'$ such that $\sigma(A,B,s)=(A,B,s)$. Then $g=\sigma^{-1} f' \sigma$ is an automorphism of $M$ that fixes $\mathrm{acl}_M(\sigma^{-1}(A))=K$ and moves $\sigma(s)=s$ out of $\mathrm{acl}(\sigma^{-1}(B))=S$.

    \noindent Using the same argument, there exists $h \in G_{(S)}$ such that $h(g(s)) \notin g(S)$.
    Therefore $g^{-1}hg(s) \notin S$, i.e.,  $g^{-1}hg \notin G_{\{S\}}$.
\end{proof}

\begin{remark}
For the rest of the paper, we could use Proposition~\ref{3petals} with Galois-algebraicity instead of saturation and model theoretic algebraicity. However, since the latter is standard in the literature, we follow that method.
\end{remark}

\begin{proposition}\label{unique_K}
    The support of $H$ is unique.
\end{proposition}
\begin{proof}
    Let $K_1,K_2\in \mathbf{A}(M)$ such that for $i=1,2$ we have $$G_{(K_i)} \leq H \leq G_{\{K_i\}}.$$
    We prove that $K_1 \subseteq K_2$; the converse follows analogously. By the inequalities, we have $G_{(K_1)} \leq H \leq G_{\{K_2\}}.$
    Suppose $K_2 \nsubseteq K_1$ then by Proposition \ref{3petals} there are $g \in G_{(K_1)}$ and $h \in G_{(K_2)}$ such that $g^{-1}hg \notin G_{\{K_2\}}$, a contradiction since $h \in G_{(K_2)} \subseteq G_{\{K_2\}}$ and $g \in G_{(K_1)} \subseteq G_{\{K_2\}}$ by the inequalities.
\end{proof}

\begin{definition}\label{GS_def} The \emph{generalized pointwise stabilizer} of $(K, L) \in \mathbf{EA}(M)$ is defined as follows:
 $$ G_{(K, L)} = \{ f \in \mathrm{Aut}(M) : f \restriction K \in L \}.$$
We denote by $\mathcal{GS}(M)$ the set of all generalized pointwise stabilizers.
\end{definition}

Notice that if $L = \{ \mathrm{id}_K \}$, then  $G_{(K, L)} = G_{(K)}$, i.e., it equals the pointwise stabilizer of $K$, and that if $L = \mathrm{Aut}_M(K)$, then $G_{(K, L)} = G_{\{ K \}}$ ({cf.} \ref{algclosednotation}(4)), i.e., it equals the setwise stabilizer of $K$.
That is why we referred to $ G_{(K, L)}$ as a generalized pointwise stabilizer.

\begin{lemma}\label{baguettelemma} If $H \leq G$ is such that $G_{(K)} \leq H \leq G_{\{K\}}$ for a certain $K \in \mathbf{A}(M)$, then $H \in \mathcal{GS}(M)$. In particular, if $H \leq G$ open, then $H \in \mathcal{GS}(M)$.  
\end{lemma}

\begin{proof} Since every automorphism in $\mathrm{Aut}_M(K)$ extends to an automorphism of $M$ ({cf.} \ref{algclosednotation}(4)) we have the following isomorphism:
\begin{align*}\label{equation_label}
    f:G_{\{K\}}/G_{(K)} &\cong \mathrm{Aut}_M(K) \\
    gG_{(K)} &\mapsto g\restriction K
\end{align*}
By the fourth isomorphism theorem, $H=G_{(K,L)}$, for $L=\{f\restriction K : f \in H\}$.
\end{proof}

 \begin{proposition}\label{normality}
Let $H_1, H_2 \in \mathcal{GS}(M)$. If $H_1 \trianglelefteq H_2$, then there are $K \in \mathbf{A}(M)$ and $L_1 \leq L_2$ such that $H_i = G_{(K, L_i)}$ for $i =1,2$.
\end{proposition}

\begin{proof}
Firstly, $K_2 \subseteq K_1$.
Suppose not, by Proposition \ref{3petals} there are $g \in G_{(K_1)}$ and $h \in G_{(K_2)}$ such that $g^{-1}hg \notin G_{\{K_1\}}$. The automorphism $g \in G_{(K_1)}\subseteq H_1$ but $g \notin G_{\{K_2\}}$, and so $g \in H_1\setminus H_2$, a contradiction.

\noindent Secondly, $K_1 \subseteq K_2$. Suppose not, by Proposition \ref{3petals} there are $g \in G_{(K_2)}$ and $h \in G_{(K_1)}$ such that $g^{-1}hg \notin G_{\{K_2\}}$. Since $h \in G_{(K_1)} \subseteq H_1$ and $g \in G_{(K_2)} \subseteq H_2$, $H_1 \not\trianglelefteq H_2$, a contradiction.

\noindent Finally, $L_1 \leq L_2$. Suppose not, and let $h \in L_1 \setminus L_2$. Then, recalling Notation \ref{algclosednotation}(4), $h$ extends to an automorphism $\hat{h}$ of $M$. Clearly $\hat{h} \in H_1 \setminus H_2$, a contradiction.
\end{proof}

In general, $\mathcal{GS}(M)$ lies between the open and the $G_\delta$ subgroups. We need the following subset to characterize the pointwise stabilizers.

\begin{definition}
    The \emph{open limits} of $G$ are defined as follows:
    $$\mathcal{OL}(M)=\{H\leq G: H=\bigcap_{i <\omega}H_i \text{ with } H_{i+1}\trianglelefteq H_i \text{ and } H_i \text{ are open }\forall i < \omega \}$$
\end{definition}

\begin{proposition}\label{OLinGS} $\mathcal{OL}(M) \subseteq \mathcal{GS}(M)$. 
\end{proposition}

\begin{proof}
    Let $H \in \mathcal{OL}$. Then $H=\bigcap_{i <\omega}H_i$ with $H_i$ open subgroups for every $ i \in \omega$. By Lemma \ref{baguettelemma} for every $i \in \omega $ there exists $(K_i,L_i)\in \mathbf{EA}(M)$  such that $H_i=G_{(K_i,L_i)}$. Since $H_{i+1} \trianglelefteq H_i$, by Proposition \ref{normality}, for every $i \in \omega$ $K_i=K$  for some $K \in \mathbf{A}(M)$. We claim that $H=G_{(K,L)}$ with $L=\{g \in {\mathrm{Aut}_M(K)}: g \in L_i \text{ for every } i < \omega \}$. Let $f \in H$. Since $f \in G_{(K,L_i)}$, $f\restriction K \in L_i$ and that is $f \in G_{(K,L)}$. For the other containment, suppose not and let $f \in G_{(K,L)} \setminus H$. Then $f \notin G_{(K,L_i)}$ for some $i <\omega$. So  $f\restriction K \notin L_i$, that is $f \notin L$, a contradiction. 
\end{proof}

\begin{proposition}\label{PSinLO} $\mathcal{PS}(M)\subseteq \mathcal{OL}(M)$. 
\end{proposition}
\begin{proof}
    Let $K \in \mathbf{A}(M)$, i.e., there exists $A \subseteq M$ finite such that $K=\mathrm{acl}_M(A)$. We want to define an increasing chain $\{A_i\}_{i \in I}$ such that the respective stabilizers form a decreasing chain of open subgroups. Let $\{b_i\}_{i \in I}$ a countable enumeration of $K$. We set $A_0=A$. Since $b_0 \in K$, the orbit $\mathcal{O}(b_0/A)$ is finite. Let $A_1=A_0\cup \mathcal{O}(b_0/A)$. Therefore $G_{(A_1)}$ {is an open subgroup} and $G_{(A_1)} \trianglelefteq G_{(A_0)}$. In the same way, we can define $A_{i+1}=A_i \cup \mathcal{O}(b_i/A)$ and we have $G_{(A_{i+1})}$ is an open subgroup and $G_{(A_{i+1})}\trianglelefteq G_{(A_i)}$. Clearly, $G_{(K)}=\bigcap_{i \in I}  G_{(A_i)}$, i.e.,  $G_{(K)} \in \mathcal{OL}(M)$.
\end{proof}

The next proposition is the core of this section and provides a characterization of the pointwise stabilizers.

 \begin{proposition}\label{char_point_stab}  
 $\mathcal{PS}(M) = \{ H \in \mathcal{OL}(M) : \not\!\exists H' \in \mathcal{OL}(M) \text{ such that } H' \lhd H  \}.$
\end{proposition}

 \begin{proof} Let $H \in \mathcal{PS}(M)$ and assume that there exists $H' \in \mathcal{OL}(M)$ such that $H' \lhd H$. By Proposition \ref{OLinGS}, $H' = G_{(K', L')}$ for $(K', L') \in \mathbf{EA}(M)$. 
 By Proposition \ref{normality},  $K' = K$ and $L' \leq \{ \mathrm{id}_K \}$. So $H' = H$, a contradiction.
 
\noindent Conversely, let $H$ as above, then $H = G_{(K, L)}$ for $(K, L) \in \mathbf{EA}(M)$ by {Proposition}~\ref{OLinGS}. If $L \neq \{ \mathrm{id}_K \}$ then letting $H' = G_{(K)}$ we have $H' \lhd H$, a contradiction since $H' \in \mathcal{OL}(M)$ by Proposition~\ref{PSinLO}.
\end{proof}

Once we have characterized the pointwise stabilizers, we aim to refine this description by obtaining a layer-by-layer characterization. We first define a layering on the lattice of finitely generated algebraically closed sets.
\begin{notation}\label{the_lattice notation} Let $\mathbf{A}_+(M)$ be the set $\mathbf{A}(M)$ extended with an extra element $\mathrm{dcl}^{\mathrm{g}}_M(\emptyset)$ (of course if $\mathrm{dcl}^{\mathrm{g}}_M(\emptyset) = \mathrm{acl}_M(\emptyset)$, then $\mathbf{A}_+(M) =\mathbf{A}(M)$). On $\mathbf{A}_+(M)$ we consider a poset structure by letting $0 = \mathrm{dcl}^{\mathrm{g}}_M(\emptyset) \leq \mathrm{acl}_M(\emptyset)$ be at the bottom of the lattice and by defining (as usual) $K_1 \leq K_2$ if and only if $K_1 \subseteq K_2$. For every $0 < k < \omega$, we define $\mathbf{A}_k(M)$ as the set of $K \in \mathbf{A}(M)$ such that $\mathrm{dcl}^{\mathrm{g}}_M(\emptyset) \neq K$ and if there are $\mathrm{dcl}^{\mathrm{g}}_M(\emptyset) \lneq K_1 \lneq K_2 \lneq \cdots \lneq K_m = K$, then $m \leq k$. Also, by convention, we let $\mathbf{A}_\omega(M) = \mathbf{A}(M)$.
\end{notation}
 Once we have characterized $\mathcal{PS}(M)$, it is immediate to characterize those elements of $\mathcal{PS}(M)$ which correspond to $\mathbf{A}_k(M)$, recalling Notation~\ref{the_lattice notation} and {Proposition} \ref{char_point_stab}.

 \begin{proposition}\label{char_point_stab_singletons} For $0 < k < \omega$, let
 $\mathcal{PS}_k(M)$ be the set of  $H \in \mathcal{PS}(M)$ such that $H \neq G$ and any chain of elements of $\mathcal{PS}(M)$ starting at $H$ and arriving at $G$ has length $\leq k$. Also, by convention, we let $\mathcal{PS}_\omega(M) = \mathcal{PS}(M)$. Then, recalling the notation introduced in \ref{the_lattice notation} (so $\mathbf{A}_\omega(M) = \mathbf{A}(M)$), we have the following:
 $$\mathcal{PS}_k(M) = \{G_{(K)} : K \in \mathbf{A}_k(M) \}.$$
\end{proposition}

 \section{Reconstructing isomorphisms}\label{sec_reconstruction}

In this section, we show that, with the necessary adjustments, the reconstruction theorem from \cite{Paolini_BLMS} remains valid in our context. This, in turn, allows us to apply it in Section \ref{sec_applications} to structures that are not $\omega$-categorical.

\begin{notation}\label{kM_notation} We define $k_M$ as the following cardinal:
 $$k_M = \mathrm{sup} \{k < \omega : \exists a_1, \dots, a_k \in M \text{ s.t. } \mathrm{acl}_M(a_1) \subsetneq \dots \subsetneq \mathrm{acl}_M(a_k)\}.$$
\end{notation}

We now introduce a first-order structure which we refer to as the {\em expanded structure of $M$ of depth $\leq k$} (for $k \leq \omega$). Morally, it is the orbital structure on the set of algebraically closed sets of $M$ which are at distance $\leq k$ from the bottom of the poset $(\mathbf{A}(M), \subseteq)$.

\begin{definition}\label{def_expanded_1}  Let $0 < k \leq \omega$ and recall the notation from \ref{the_lattice notation}.
\begin{enumerate}[(1)]
 \item We let $M^{\mathrm{ex}}_k$ be $\{(K, p, K') : K, K' \in \mathbf{A}_k(M) \text{ and } p: K \cong K'$  such that {there }exists $ f \in G, f\restriction K=p \}.$  
 \item\label{the_identification} We identify $\mathbf{A}_k(M)$ with the set of triples $(K, \mathrm{id}_K, K)$.
\end{enumerate}
\end{definition}

\begin{definition}\label{def_expanded_12} Let $0 < k \leq \omega$. We define a first-\mbox{order structure $\mathcal{E}^{\mathrm{ex}}_M(k)$ as follows:}
 \begin{enumerate}[(1)]
 \item $\mathcal{E}^{\mathrm{ex}}_M(k)$ has as domain $M^{\mathrm{ex}}_k$;
 \item\label{def_expanded_12_unary} we define a unary predicate which holds of $\mathbf{A}_k(M)$, recalling \ref{def_expanded_1}(\ref{the_identification}); 
 \item\label{def_expanded_12_binary} for $1 < n< \omega$, we define an $n$-ary predicate $E_n$ as follows:
 $$((A_1, p_1, B_1), ..., (A_n, p_n, B_n)) \in E^{\mathcal{E}^{\mathrm{ex}}_M(k)}_n$$
 $$\Updownarrow$$
 $$\exists g \in \mathrm{Aut}(M) \text{ s.t. for all } i \in [1, n], \,  g(A_i) = B_i \text{ and } g \restriction A_i = p_i;$$
 \item\label{dom} we define a binary predicate $\mathrm{Dom}$ (which stands for ``domain'')~such~that \newline $(A, p, B) \in \mathcal{E}^{\mathrm{ex}}_M(k)$ is in relation with $C \in \mathbf{A}_k(M)$ if and only if $A = C$;
 \item\label{cod} we define a binary predicate $\mathrm{Cod}$ (which stands for ``codomain'') such that $(A, p, B) \in \mathcal{E}^{\mathrm{ex}}_M(k)$ is in relation with $C \in \mathbf{A}_k(M)$ if and only if $B = C$.
\end{enumerate}
When we write $\mathcal{E}^{\mathrm{ex}}_M$ (so without the $k$) we mean $\mathcal{E}^{\mathrm{ex}}_M(\omega)$.
\end{definition}

Although the definition of the structure is rather technical, each of its predicates is required for the following theorem. Among the predicates involved, the relations $E_n$ play the central role, as they express precisely when partial isomorphisms between algebraically closed subsets can be amalgamated.

\begin{fact}
As pointed out in \cite{Paolini_BLMS} the following predicates are definable in $\mathcal{E}^{\mathrm{ex}}_M(k)$:
\begin{itemize}
\item a ternary predicate which holds exactly when the following happens:
$$K_1 \xrightarrow{p_1} K_2 \xrightarrow{p_2} K_3 = K_1 \xrightarrow{p_2 \circ p_1} K_3.$$
\item a binary predicate which holds exactly when the following happens:
$$K_1 \xrightarrow{p_1} K_2 \xrightarrow{p_2} K_1 = K_1 \xrightarrow{\mathrm{id}_{K_1}} K_1.$$
\end{itemize}
\end{fact}

 The problem of determination of $\mathrm{Aut}(M) \cong_{\mathrm{top}} \mathrm{Aut}(N)$ reduces completely to the problem of determination of $\mathcal{E}^{\mathrm{ex}}_M(k) \cong \mathcal{E}^{\mathrm{ex}}_N(k)$ as soon as $k \geq \mathrm{max}\{k_M, k_N\}$.

 \begin{theorem}\label{main_theorem} Let $M, N$ be countable saturated structures with the Lascar Property. Then
 \begin{enumerate}[(1)]
 \item if $\mathrm{Aut}(M) \cong_{\mathrm{top}} \mathrm{Aut}(N)$ then $\mathcal{E}^{\mathrm{ex}}_M(k) \cong \mathcal{E}^{\mathrm{ex}}_N(k)$, for all $0 < k \leq \omega$;
 \item if $\mathrm{max}\{k_M, k_N\} \leq k$, then $\mathcal{E}^{\mathrm{ex}}_M(k) \cong \mathcal{E}^{\mathrm{ex}}_N(k)$ implies $\mathrm{Aut}(M) \cong_{\mathrm{top}} \mathrm{Aut}(N)$;
        \item {for} $k_M \leq k \leq \omega$  
        \begin{align*}
            \mathrm{Aut}_{\mathrm{top}}(\mathrm{Aut}(M)) &\cong \mathrm{Aut}(\mathcal{E}^{\mathrm{ex}}_M(k))\\
            \alpha &\mapsto f_{\alpha}
        \end{align*}

 \end{enumerate}
\end{theorem}

\begin{proof} Modulo minor adjustments, the proof of this is as in \cite[Theorem 1.3]{Paolini_BLMS}.
\end{proof}

\begin{theorem1.2} Let $M, N$ be countable saturated structures with the Lascar Property. Then
 \begin{enumerate}[(1)]
 \item pointwise stabilizers of finitely generated algebraically closed subsets of $M$ are $G_\delta$-subgroups of $\mathrm{Aut}(M)$, and there is a Galois correspondence between them and the set of finitely generated algebraically closed subsets of $M$;
 \item the structure $\mathcal{E}^{\mathrm{ex}}_M$ associated to $M$ as in Definition \ref{def_expanded_12} is such that: $$\mathrm{Aut}(M) \cong_{\mathrm{top}} \mathrm{Aut}(N) \; \Leftrightarrow \; \mathcal{E}^{\mathrm{ex}}_M \cong \mathcal{E}^{\mathrm{ex}}_N;$$
 and $$\mathrm{Aut}_{\mathrm{top}}(\mathrm{Aut}(M)) \cong \mathrm{Aut}(\mathcal{E}^{\mathrm{ex}}_M).$$
 \end{enumerate}
 \end{theorem1.2}

\begin{proof}
We prove item (1). By Proposition \ref{char_point_stab}, the pointwise stabilizers $\mathcal{PS}(M)$ are characterized, purely group-theoretically, among the topological subgroups of~$G$. Indeed, the pointwise stabilizers are the minimal intersections of normal chains of open subgroups under the normal subgroup relation.

 We define the following maps:
\begin{align*}
    G_{(\cdot)}:\mathbf{A}(M) &\rightarrow \mathcal{PS}(M) \quad\quad&
    \mathrm{Fix}_G:\mathcal{PS}(M)&\rightarrow \mathbf{A}(M)
\end{align*}
where $\mathrm{Fix}_G(H)=\{m \in M: \forall h \in H (h(m)=m)\}$ for all $H \in \mathcal{PS}(M)$.

\noindent The map $\mathrm{Fix}_G$ is well defined and $\mathrm{Fix}_G(G_{(K)})=K$. Indeed, $\mathrm{Fix}_G(G_{(K)})\supseteq K$ is clear. Conversely, let $m \in \mathrm{Fix}_G(G_{(K)})$, then $\mathcal{O}(m/K)=\{m\}$, thus finite. Since $K$ is algebraically closed, $m \in K$. Therefore,  $\mathrm{Fix}_G(G_{(K)})\subseteq K$. 
Since $$G_{(\cdot)}(\mathrm{Fix}_G(G_{(K)}))=G_{(K)}\text{ and } \mathrm{Fix}_G(G_{(K)})=K,$$ we have a Galois correspondence between $\mathcal{PS}(M)$ and $\mathbf{A}(M)$.
\medskip

\noindent
Item (2) is a consequence of Theorem \ref{main_theorem} with $k=\omega$.
\end{proof}

\begin{theorem}\label{no_alge} Let $M$ be a countable saturated structure with the Lascar Property and such that $\mathrm{acl}_M(a)=\{a\}$ for every $a \in M$, then we have:
 $$\mathrm{Aut}_{\mathrm{top}}(\mathrm{Aut}(M)) \cong \mathrm{Aut}(\mathcal{E}_M),$$
where $\mathcal{E}_M$ is the orbital structure associated to $M$. In particular, if $M$ has in addition the small index property, then $\mathrm{Aut}(\mathrm{Aut}(M)) \cong \mathrm{Aut}(\mathcal{E}_M)$. 
\end{theorem}

\begin{proof} The argument is essentially the same as in the proof of \cite[Theorem 1.4]{Paolini_BLMS}, noting that the implication of the Ahlbrandt--Ziegler Theorem used there does not require the assumption of $\omega$-categoricity (see \cite[p.~226, 5.4.8(b)]{hodges}).

\end{proof}

\section{When are our assumptions satisfied?}\label{section4}

In this section, we identify sufficient conditions under which the Lascar Property holds.
Throughout this section, we write $AB$ to denote  $A \cup B$. 

\begin{definition} \label{rkdef}
    Let $\mathrm{rk}$ be an integer-valued function defined on $\mathcal{P}(M)$. 
    We say that $\mathrm{rk}$ is a dimension function on $M$ if, for all $A,B \in \mathbf{A}(M)$ and $C\subseteq M$ we have:
    \begin{enumerate}[(1)]
        \item $\mathrm{rk}(g(C))=\mathrm{rk}(C)$ for all $g \in G$;
        \item $0\leq \mathrm{rk}(A) \leq \mathrm{rk}(AB)\leq \mathrm{rk}(A) + \mathrm{rk}(B)- \mathrm{rk}(A\cap B)$;
        \item If $A \subseteq B$ and $\mathrm{rk}(A)=\mathrm{rk}(B)$, then $A=B$;
        \item $\mathrm{rk}(A)<\omega$.
    \end{enumerate} 
    If $A,B \subseteq M$, then we define $\mathrm{rk}(A/B)=\mathrm{rk}(AB)-\mathrm{rk}(B)$. 
\end{definition}
 
We will need the notion of a stationary independence relation from \cite{stationary_indep}.

\begin{definition}\label{stationary}
We say that the ternary relation $\ind$ on $\mathcal{P}(M)$ is a stationary independence relation compatible with $\mathrm{acl}_M$ if for every $A,B,C,D \in \mathbf{A}(M)$, $\mathrm{acl}_M$-closed set $E$ and finite tuples $a,b \in M^{< \omega}$ the following conditions are met:
    \begin{enumerate}[(1)]
        \item (Compatibility) We have $a\ind_b C \Leftrightarrow a\ind_{\mathrm{acl}_M(b)} C$ and $$a\ind_B C \Leftrightarrow e \ind_B C  \text{ for all } e \in \mathrm{acl}_M(a,B)\Leftrightarrow \mathrm{acl}_M(a,B)\ind_B C.$$
        \item (Invariance) If $g \in G$ and $A\ind_B C $, then $g(A)\ind_{g(B)} g(C)$.
        \item (Monotonicity) If $A\ind_B CD$, then $A\ind_B C$ and $A\ind_{BC} D$ and if $A\ind_B CE$, then $A\ind_B C$.
        \item (Transitivity) If $A\ind_B C$ and $A\ind_{BC} D$, then $A \ind_B CD$.
        \item (Symmetry) If $A\ind_B C$, then $C \ind_B A$.
        \item (Existence) There is $g \in G_{(B)}$ with $g(A) \ind_B C$.
        \item (Stationary) Suppose $B\subseteq A,D$, $A\ind_B C$ and $D\ind_B C$. If $h: A \rightarrow D$ is the identity of $B$ and extends to an automorphism of $M$, then there is some $k \in \mathrm{Aut}(M)$ such that $k\supseteq h \cup \mathrm{id}_C$.
    \end{enumerate}
\end{definition}

    \begin{definition} We define $A \rkindep[B] C \Leftrightarrow \mathrm{rk}(A/BC)=\mathrm{rk}(A/B)$.
\end{definition}

    \begin{convention} Below when we write ``Suppose that $\rkindep$ is a stationary independence relation'' we mean ``$\rkindep$ is a stationary independence relation compatible with $\mathrm{acl}_M$''.
\end{convention}

\begin{definition}\label{def_WCB}
    Let $A, B \in \mathbf{A}(M)$, we say that $C$ is the weak canonical base for the pair $(A, B)$ if it is the smallest set in $\mathbf{A}(M)$ such that $A\ind_C B$ and we denote it by $\mathrm{cb}(A/B)$ (so in particular when we write this we implicitly assume that such a $C$ exists). We say that $M$ has weak canonical bases (for $\ind$) if every pair $(A,B)$, where $A \in \mathbf{A}(M)$ and $B$ is an $\mathrm{acl}_M$-closed set, admits a weak canonical base.
\end{definition}

\begin{theorem}\label{lascar_ind_th} Suppose $\rkindep$ is a stationary independence relation.
    \begin{enumerate}[(1)]
        \item For every $A, B \in \mathbf{A}(M)$, if $A \rkindep_{A \cap B} B$, then $G_{(A\cap B)}=\langle G_{(A)} \cup G_{(B)} \rangle$.
    
        \item Suppose (in addition) that $M$ has weak canonical bases. Then, for every $A, B \in \mathbf{A}(M)$, $G_{(A\cap B)}=\langle G_{(A)} \cup G_{(B)} \rangle$.
    \end{enumerate}
\end{theorem}
\begin{proof} 
        We prove item (1). We show the containment from left to right, since the other one is immediate. Let $g \in G_{(A\cap B)}$. By Existence, there is $ h\in 
        G_{(A \cap B)}$ such that:
        \begin{equation*}\tag{$\star$}\label{equation_label1}
            h(A)\textstyle\rkindep_{A\cap B} ABg(A). 
        \end{equation*}
        The automorphism $h: A \mapsto h(A)$ fixes $A\cap B$, $A \rkindep_{A \cap B} B$ by hypothesis and $h(A) \rkindep_{A \cap B} B$ by (\ref{equation_label1}) and Monotonicity. So, by Stationarity, there is $k \supseteq h \cup \mathrm{id}_B$, i.e., $k \in G_{(B)}$ and $k(A)=h(A)$. The subgroup $G_{(K)} \leq \langle G_{(A)} \cup G_{(B)} \rangle$ where $K=k(A)$. Indeed, every $f \in G_{(K)}$, can be written as:
        \begin{equation*}
            f= \underbrace{k}_{\in G_{(B)}} \underbrace{k^{-1} f k}_{\in G_{(A)}} \underbrace{k^{-1}}_{\in G_{(B)}}.
        \end{equation*}
        The automorphism $g: A \mapsto g(A)$ fixes $A\cap B$, and the following two independence relations $A \rkindep_{A \cap B} k(A)$ and $g(A) \rkindep_{A \cap B} k(A)$ are guaranteed by (\ref{equation_label1}), $h(A)=k(A)$, Symmetry, and Monotonicity. 
        So, by Stationarity, there is $t \supseteq g \cup \mathrm{id}_K$, i.e., $t \in G_{(K)}$ and $t(A)=g(A)$. The automorphism $t^{-1}g \in G_{(A)}$, therefore $g \in \langle G_{(A)} \cup G_{(B)} \rangle$.

    \medskip \noindent
        Concerning item (2), this is by induction on $n=\mathrm{rk}(A/A\cap B) - {\mathrm{rk}}(A/B)$. The base case $n=0$ implies $A \rkindep_{A \cap B} B$, so it holds by (1). 
        Let $H=\langle G_{(A)} \cup G_{(B)} \rangle$. Suppose $A \nrkindep_{A \cap B} B$ and let $C=\mathrm{cb}(A/B)$. By Existence, $A \rkindep_B B$, so, by definition $C\subseteq B$ and $A\rkindep_C B$. By Existence, there is $g \in G_{(B)}$ such that $g(A) \rkindep_B AB$. Let $A'=g(A)$ and note that $A\cap B=A' \cap B$. By Invariance, $g(A) \rkindep_{g(C)} g(B)$, i.e., $A'\rkindep_C B$ and, by Transitivity, $A'\rkindep_C B$ and $A'\rkindep_{BC} AB$ imply $A'\rkindep_C AB$.

        \noindent Note that we have $A\cap B\subseteq A \cap A'$. Let $D=\mathrm{cb}(A/A'\cap A)$ that, as above, $D \subseteq A$. From $g(A) \rkindep_B AB$ , by Monotonicity, $A\rkindep_B A'\cap B$. So $D\subseteq B$, in particular, $D\subseteq A\cap B$. Using $A\cap B\subseteq A\cap A' $ and Monotonicity, $A\rkindep_C A'\cap A$ implies $A\rkindep_{A\cap B} A'\cap A$. In terms of rank, this implies $\mathrm{rk}(A\cap A')=\mathrm{rk}(A\cap B)$. By \ref{rkdef}(3) we have $A\cap B = A \cap A'$.

        \noindent Suppose $A'\rkindep_A B$. By Transitivity, $A'\rkindep_A AB$ and hence we have that $C':=\mathrm{cb}(A'/\mathrm{acl}_M(AB)) \subseteq A$. But also for $C'\subseteq C \subseteq B$. So $C'\subseteq A\cap B=A'\cap B$ and by Monotonicity, $A' \rkindep_{A'\cap B} B$. By Invariance, we have $A\rkindep_{A\cap B} B$, which contradicts what we are assuming.

        \noindent So we know that $A'\nrkindep_A B$ and, by Symmetry, $B\nrkindep_A A'$. In terms of the rank, this implies
        $$\mathrm{rk}(B/AA')<\mathrm{rk}(B/A)$$ so by definition $$\mathrm{rk}(AA'B)-\mathrm{rk}(AA') < \mathrm{rk}(AB)-\mathrm{rk}(A).$$ Using $\mathrm{rk}(A)=\mathrm{rk}(A')$ we have
        $$\mathrm{rk}(A)+\mathrm{rk}(A')-\mathrm{rk}(AA')< \mathrm{rk}(AB)-\mathrm{rk}(AA'B)+\mathrm{rk}(A).$$
        The right hand side is $$-\mathrm{rk}(A'/AB) +\mathrm{rk}(A)= \mathrm{rk}(A) - \mathrm{rk}(A'/B)=\mathrm{rk}(A)-\mathrm{rk}(A/B)$$ using the fact that $A'\rkindep _B A$ and $\mathrm{rk}(AB)=\mathrm{rk}(A'B)$. So putting this together and using $A\cap B = A \cap A'$:
        $$\mathrm{rk}(A)+\mathrm{rk}(A')-\mathrm{rk}(A\cap A') -\mathrm{rk}(AA') < \mathrm{rk}(A/A\cap B)-\mathrm{rk}(A/B) =n. $$ Thus $\mathrm{rk}(A/A\cap A') - \mathrm{rk}(A/A') <n$. As $G_{(A)}$,$G_{(A')}\leq H$ and $A\cap A'=A\cap B$, we can conclude by induction that $H=G_{(A\cap B)}$.
\end{proof}

\begin{theorem}\label{th_Lascar_baguette}
    Let $M$ be a countable structure satisfying the following conditions:
    \begin{enumerate}[(1)]
        \item there is no strictly decreasing sequence $A_0 \supsetneq A_1 \supsetneq \dots \supsetneq A_n \supsetneq \dots$ where each $A_n \in \mathbf{A}(M)$.
        \item for every $A, B \in \mathbf{A}(M)$, $G_{(A\cap B)}=\langle G_{(A)} \cup G_{(B)} \rangle$.
\end{enumerate}
Then $M$ has the Lascar Property.
\end{theorem}

\begin{remark}
    If $\rkindep$ is a stationary independence relation, the hypotheses of the Theorem hold. Indeed, the first is a consequence of \ref{rkdef}(3) and (4), and the second is \ref{lascar_ind_th}.
\end{remark}

\begin{proof}
Let $H \leq G$ be open, then there is a finite $A$ such that $G_{(A)} \leq H$ (see \cite[4.1.5(b)]{hodges}). Hence, $G_{(\mathrm{acl}_M(A))} \leq G_{(A)} \leq H$, i.e., $H$ contains the pointwise stabilizer of at least one finite-dimensional algebraically closed set of $M$. Let $K$ be of the smallest dimension with respect to this property, i.e., being contained in $H$. Note that the first hypothesis guarantees the existence of such a set.

\noindent Let $h \in H$ and $K'=h(K)$. Then $G_{(K')} = hG_{(K)}h^{-1} \leq H$ (as $h \in H$ and $G_{(K)} \leq H$). By \ref{rkdef}(1), the ranks of $K$ and $K'$ coincide. Since $G_{(K)} \leq H$ and $G_{(K')} \leq H$, by the second hypothesis, $G_{(K \cap K')} \leq H$. Therefore, $\mathrm{rk}(K \cap K')$ has to have dimension $k:= \mathrm{rk}(K)$, as otherwise, we would contradict the minimality of the dimension of $K$.  
Hence, by \ref{rkdef}(2), we have the following situation:
\begin{equation*}
    \mathrm{rk}(KK') \leq \mathrm{rk}(K) + \mathrm{rk}(K') - \mathrm{rk}(K \cap K') = k+k-k = k.
\end{equation*}
Thus $\mathrm{rk}(KK') = k$, therefore by \ref{rkdef}(3), we have $K = K'$. So, recalling that $K' = h(K)$, we have $h \in G_{\{K\}}$.

\end{proof}

\section{Applications} \label{sec_applications}

This final section demonstrates the framework's broad applicability by addressing both $\omega$-categorical and strongly minimal theories.
\begin{remark}
    Every $M$ $\omega$-categorical is saturated (see \cite[4.3.4]{tentziegler}) and the property of weak elimination of imaginaries is equivalent to the Lascar Property in the $\omega$-categorical framework (see e.g. \cite{Paolini_BLMS} and references therein). Thus, our framework generalizes the framework of \cite{Paolini_BLMS}.
\end{remark}

The discussion now turns to the strongly minimal case, which serves as the primary motivation for this work.

\begin{definition}
    We say that $(X,\mathrm{cl})$ is a pregeometry if $X$ is a set and $\mathrm{cl}: \mathcal{P}(X) \rightarrow \mathcal{P}(X)$ is a function satisfying the following five properties for every subsets $A,B$ of $X$ and $a,b \in X$:
    \begin{enumerate}[(1)]
        \item (Reflexivity) $A\subseteq \mathrm{cl}(A)$;
        \item (Monotonicity) if $A\subseteq B$ then $\mathrm{cl}(A)\subseteq \mathrm{cl}({B})$;
        \item (Finite Character) if $a \in \mathrm{cl}(A)$  then $a \in \mathrm{cl}(A')$ for a finite $A'\subseteq A$;
        \item (Idempotency) $\mathrm{cl}(\mathrm{cl}(A))=\mathrm{cl}(A)$;
        \item (Exchange Property) if $a \in \mathrm{cl}(A \cup \{b\})$ but $a \notin \mathrm{cl}(A)$, then $b \in \mathrm{cl}(A \cup \{a\})$.
    \end{enumerate}
    Furthermore, $(X,\mathrm{cl})$ is a geometry if, for every $a \in X$, we have $\mathrm{cl}(\{a\})=\{a\}$.
\end{definition}
\begin{definition}
    An automorphism of the pregeometry (or a geometry)  $(X,\mathrm{cl})$ is a bijection $f:X\rightarrow X$ such that $f(\mathrm{cl}(A))=\mathrm{cl}(f(A))$  for any $A\subseteq X$.
\end{definition}

If $M$ is a strongly minimal structure, then $(M, \mathrm{acl}_M)$ is a pregeometry (see e.g. \cite[5.7.5]{tentziegler}). 
Letting Morley rank be our dimension function and non-forking as our $\rkindep$ relation, the hypotheses from Section~\ref{section4} hold in a countable saturated strongly minimal structure with weak elimination of imaginaries (see {Chapters} 6 and 8 of \cite{marker2}). In particular, they hold for countable saturated algebraically closed fields and countably infinite-dimensional $\mathbb{K}$-vector spaces, for $\mathbb{K}$ a countable field. Indeed, they both have weak elimination of imaginaries; see, e.g., \cite[16.21]{poizat} for algebraically closed fields and \cite[4.4.7]{hodges} for $\mathbb{K}$-vector spaces.

\begin{remark}\label{kmpregeometriequal1}
     If $(M,\mathrm{acl}_M)$ is a pregeometry then $k_M=1$ (cf. \ref{kM_notation}). Indeed, suppose there exists $a,b \in M$ such that $\mathrm{acl}_M(\emptyset)\subsetneq \mathrm{acl}_M(a)\subseteq \mathrm{acl}_M(b)$. Since $a \in \mathrm{acl}_M(b)\setminus \mathrm{acl}_M(\emptyset)$, by Exchange Property, $b \in \mathrm{acl}_M(a)$. Therefore $\mathrm{acl}_M(b)\subseteq \mathrm{acl}_M(\mathrm{acl}_M(a))=\mathrm{acl}_M(a)$ by Idempotency. 
\end{remark}

\begin{definition}
    Let $(M,\mathrm{acl}_M)$ be a pregeometry. The canonical geometry of $(M,\mathrm{acl}_M )$ is $(\mathbb{G}_M,\mathrm{cl}_{\mathbb{G}_M})$ defined as follows:  
    \begin{itemize}
        \item $\mathbb{G}_M=\{\mathrm{acl}_M(a): a\in M\setminus \mathrm{acl}_M(\emptyset) \}$;
        \item $\mathrm{cl}_{\mathbb{G}_M}(X)=\{K \in \mathbb{G}_M: K\subseteq \mathrm{acl}_M(\bigcup X)\}$ for every $X\subseteq \mathbb{G}_M$.
    \end{itemize}
\end{definition}

\begin{lemma}\label{canonicalmaplemma}
  Let $M$ be a countable saturated structure with the Lascar Property and such that $(M,\mathrm{acl}_M)$ is a pregeometry. Then there exists a canonical map $\pi$ from $\mathrm{Aut}_{\mathrm{top}}(\mathrm{Aut}(M))$ into $\mathrm{Aut}(\mathbb{G}_M)$.
\end{lemma}

\begin{proof}
    By Remark \ref{kmpregeometriequal1} and Theorem \ref{main_theorem}(3), there is a well defined  group isomorphism $\psi$ from $\mathrm{Aut}_{\mathrm{top}}(\mathrm{Aut}(M))$ into $\mathrm{Aut}(\mathcal{E}^{{\mathrm{ex}}}_M(1))$.

    \noindent Let $\alpha \in \mathrm{Aut}(\mathcal{E}^{{\mathrm{ex}}}_M(1))$. Recalling Proposition \ref{char_point_stab_singletons} and Theorem \ref{main_theorem}, we can  define a map $\xi_\alpha$ from $\mathbb{G}_M$ onto itself by: for every $K \in \mathbb{G}_M$ $\xi_\alpha(K)=K'$ where $\alpha(K,\mathrm{id}_K,K)=(K',\mathrm{id}_{K'},K')$. It is easy to see that $\xi_\alpha \in \mathrm{Aut}(\mathbb{G}_M)$ and the map $\xi$ from $\mathrm{Aut}(\mathcal{E}^{{\mathrm{ex}}}_M(1))$ into $\mathrm{Aut}(\mathbb{G}_M)$ defined by $\xi(\alpha)=\xi_\alpha$ is a group homomorphism.

    \noindent Then $\pi=\xi\circ\psi$ is the natural group homomorphism from $\mathrm{Aut}_{\mathrm{top}}(\mathrm{Aut}(M))$ into $\mathrm{Aut}(\mathbb{G}_M)$ we are looking for.
\end{proof}

\begin{theorem1.3} Let $M$ be a countable saturated structure with the Lascar Property and suppose that $(M,\mathrm{acl}_M)$ is a pregeometry.

Suppose there is $\mathrm{Aut}(M)\leq H \leq \mathrm{Sym}(\omega)$ such that:
    \begin{enumerate}[(1)]
    \item $\varphi_H:H \rightarrow \mathrm{Aut}(\mathbb{G}_M)$ defined as $\varphi_H(h)(A)=h(A)$ is a well-defined homomorphism.
            \item $\gamma: H \rightarrow \mathrm{Aut}_{\mathrm{top}}(\mathrm{Aut}(M))$ defined as $h \mapsto h(\cdot)h^{-1}$ is well defined;
      \item there exists {a} homomorphism $\hat{\cdot}:\mathrm{Aut}(\mathbb{G}_M) \rightarrow H$ such that $\pi(\gamma(\hat{g}))=g$ where $\pi$ is the canonical map defined in Lemma \ref{canonicalmaplemma}.
    \end{enumerate} 
Then   $$\mathrm{Aut}_{\mathrm{top}}(\mathrm{Aut}(M))\cong \mathrm{ker}(\pi) \rtimes \mathrm{Aut}(\mathbb{G}_M)$$
and the following diagram commutes:
    $$\begin{tikzcd}
    H \arrow[r, "\varphi_H"] \arrow[d, "\gamma"]
    & \mathrm{Aut}(\mathbb{G}_M) \\
    \mathrm{Aut}_{\mathrm{top}}(\mathrm{Aut}(M)). \arrow{ur}{\pi} 
    &
    \end{tikzcd}$$
    \end{theorem1.3}
    
\begin{proof}
    Recall $\pi$, $\psi$ and $\xi$ introduced in the proof of Lemma \ref{canonicalmaplemma}.
    Then, easily, the diagram  $$\begin{tikzcd}
    H \arrow[r, "\varphi_H"] \arrow[d, "\gamma"]
    & \mathrm{Aut}(\mathbb{G}_M) \\
    \mathrm{Aut}_{\mathrm{top}}(\mathrm{Aut}(M)) \arrow{ur}{\pi} \arrow[r,"\psi" ,"\cong"']
    & \mathrm{Aut}(\mathcal{E}^{{\mathrm{ex}}}_M(1)) \arrow[u, "\xi" ]
    \end{tikzcd}$$ commutes. Indeed, let $h \in H$. Then for all $K \in \mathbb{G}_M$ $\xi\psi\gamma(h)(K)$ is the element $K'$ of $\mathbb{G}_M$ such that:
    $$(K',\mathrm{id}_{K'},K')=(\psi\gamma(h))(K,\mathrm{id}_K,K)=(h(K),\mathrm{id}_{h(K)},h(K))$$
    recalling the proof of Theorem \ref{main_theorem}(1). Therefore, $\varphi_H(h)(K)=h(K)=(\xi\psi\gamma)(h)(K)$.

     \noindent  
    Observe now that we have the following exact sequence:
    $$1\rightarrow \mathrm{ker}(\pi) \rightarrow \mathrm{Aut}_{\mathrm{top}}(\mathrm{Aut}(M)) \overset{\pi}{\rightarrow} \mathrm{Aut}(\mathbb{G}_M). $$ We prove that it is a split exact sequence.
    By items (3), for every $g \in \mathrm{Aut}(\mathbb{G}_M)$, we have that $\pi(\gamma(\hat{g}))=g$. Hence, the map $\gamma(\hat{\cdot})$ is a section of $\pi$, so we are done.
\end{proof}

    \begin{fact}\label{notation_phi} Suppose that $M$ is such that $ (M,\mathrm{acl}_M)$ is a pregeometry. Let $\alpha \in \mathrm{Aut}(M)$. Then $\alpha$ naturally induces a map $\varphi_\alpha$ from $\mathbb{G}_M$ onto itself by: $\varphi_\alpha(K)=\alpha(K)$, and it is clear that $\varphi_\alpha \in \mathrm{Aut}(\mathbb{G}_M)$. Moreover, the map $\varphi$ from $\mathrm{Aut}(M)$ to $\mathrm{Aut}(\mathbb{G}_M)$ defined by $\varphi(\alpha)=\varphi_\alpha$ is a group homomorphism, and we call it the natural map from $\mathrm{Aut}(M)$ to $\mathrm{Aut}(\mathbb{G}_M)$.
\end{fact}

\begin{lemma}\label{phiinjective}
    Let $M$ and $H$ be as in the hypotheses of Theorem~\ref{genevansth}.
    If the map $\varphi$ from Fact \ref{notation_phi} is injective, then $\mathrm{Aut}_{\mathrm{top}}(\mathrm{Aut}(M))\cong \mathrm{Aut}(\mathbb{G}_M)$ (and if $M$ has $\mathrm{SIP}$, then ``top'' can be removed).
\end{lemma}
    
\begin{proof}
    Recalling from Lemma \ref{canonicalmaplemma} that $\pi=\xi\circ\psi$. Since $\psi$ is an isomorphism, $\pi$ is injective if and only if $\xi$ is injective; hence, to conclude, it suffices to show that $\xi$ is injective.
    Indeed, if $\mathrm{ker}(\pi)$ is trivial, by Theorem \ref{genevansth}, $\mathrm{Aut}_{\mathrm{top}}(\mathrm{Aut}(M))\cong  \mathrm{Aut}(\mathbb{G}_M)$.
    
    \noindent Let $\alpha \in \mathrm{Aut}(\mathcal{E}^{{\mathrm{ex}}}_M(1))$, and assume that $\xi_\alpha$ is the identity on $\mathbb{G}_M$ that is we have $\alpha((K,\mathrm{id}_K,K))=(K,\mathrm{id}_K,K)$. We have to prove that $\alpha$ is the identity on $\mathcal{E}^{{\mathrm{ex}}}_M(1)$. Let $(K,p,K') \in \mathcal{E}^{{\mathrm{ex}}}_M(1)$. By definition, {there} exists $g \in \mathrm{Aut}(M)$ such that $(K,p,K')=(K,g\upharpoonright K,g(K))$. Let $g_\alpha=\psi^{-1}(\alpha)$. Then $\alpha((K,p,K'))=(K,g_\alpha\upharpoonright K, g_\alpha(K))$. By the $\mathrm{Cod}$ predicate, we have that $g(K)=g_\alpha(K)$. So $g^{-1}g_\alpha(K)=K$. This means exactly that $\varphi(g^{-1}g_\alpha)$ is the identity on $\mathbb{G}_M$. By hypothesis, $\varphi$ is injective, so $g^{-1}g_\alpha$ is the identity. Therefore, $\alpha((K,g\upharpoonright K,K'))=(K,g\upharpoonright K, K')$, i.e., $\alpha$ is the identity.
\end{proof}

\begin{lemma}\label{evcompl}
    Let $M$ be a countable saturated structure with the Lascar Property and suppose that $(M,\mathrm{acl}_M)$ is a pregeometry.
    If the map $\varphi$ from Fact \ref{notation_phi} is an isomorphism, then $\mathrm{Aut}_{\mathrm{top}}(\mathrm{Aut}(M))\cong {\mathrm{Aut}(\mathbb{G}_M)}$. Moreover, every topological automorphism of $\mathrm{Aut}(M)$ is inner (and if in addition $M$ has $\mathrm{SIP}$, then ``topological" can be removed in ``topological automorphism" and $\mathrm{Aut}(\mathrm{Aut}(M))\cong {\mathrm{Aut}(\mathbb{G}_M)}$). 
\end{lemma}

\begin{proof} Taking $H=\mathrm{Aut}(M)$ and $\hat{g}=\varphi^{-1}(g)$, we are under the hypotheses of Theorem \ref{genevansth}. Indeed, the map $g\mapsto \hat{g}$ is a group homomorphism and $\hat{g}(\cdot)\hat{g}^{-1} \in \mathrm{Aut}_{\mathrm{top}}(\mathrm{Aut}(M))$. Recall the following commutative diagram:

$$\begin{tikzcd}
    \mathrm{Aut}(M) \arrow[r, "\varphi","\cong"'] \arrow[d, "\gamma"]
    & \mathrm{Aut}(\mathbb{G}_M) \\
    \mathrm{Aut}_{\mathrm{top}}(\mathrm{Aut}(M)) \arrow{ur}{\pi} \arrow[r,"\psi" ,"\cong"']
    & \mathrm{Aut}(\mathcal{E}^{{\mathrm{ex}}}_M(1)) \arrow[u, "\xi" ]
    \end{tikzcd}$$
By the commutativity of the diagram $\pi(\hat{g}(\cdot)\hat{g}^{-1})=g$. 
 
 Hence, by Lemma \ref{phiinjective}, $\xi$ is injective and  $\mathrm{Aut}_{\mathrm{top}}(\mathrm{Aut}(M))\cong {\mathrm{Aut}(\mathbb{G}_M)}$. The map $\varphi$ is surjective, so $\xi$ is also surjective. Thus $\xi$ is a group isomorphism. The map $\gamma=\psi^{-1}\xi^{-1}\varphi$ is surjective, that is, every automorphism is inner.
\end{proof}

 \begin{corollary1.4} Let $M\models \mathrm{ACF}_0$ be countable and saturated, then we have:
    $$\mathrm{Aut}(\mathrm{Aut}(M))\cong \mathrm{Aut}(\mathbb{G}_M),$$
and automorphisms of $\mathrm{Aut}(M)$ are \mbox{inner (so $\mathrm{Aut}(M)$, being centerless, is complete).}
\end{corollary1.4}

\begin{proof}
    First, concerning $\mathrm{SIP}$, see \cite[Theorem~2]{Evans_ACF}. Concerning the injectivity of the map $\varphi$, this follows from the remark soon after Lemma 2.5 of \cite{evan1995}, while the surjectivity follows from Theorem A of \cite{evan1995}.    
    \noindent By Lemma \ref{evcompl},  we have that $\mathrm{Aut}(\mathrm{Aut}(M))\cong {\mathrm{Aut}(\mathbb{G}_M)}$ and every automorphism of $\mathrm{Aut}(M)$ is inner.   
\end{proof}

\begin{remark}
    Let $\mathbb{K}$ be a field and let $V$ be a non-trivial $\mathbb{K}$-vector space. Then the map $\varphi$ from Fact \ref{notation_phi} is not injective. Indeed, let $\alpha \in \mathrm{Aut}(V)$ be such that $v \mapsto kv$ with $k \in {\mathbb{K}^\times}$ and $v \in V$.
    Then, $\varphi_\alpha$ is the identity in $\mathbb{G}_V$.
\end{remark}

\begin{definition} Let $V$ {be} a $\mathbb{K}$-vector space. A semilinear automorphism of $V$ is a bijection $f$ from $V$ into itself for which there exists $\tau \in \mathrm{Aut}(\mathbb{K})$ such that, for every $c_1,c_2 \in \mathbb{K}$ and $v_1,v_2 \in V$, 
$f(c_1v_1+c_2v_2)=\tau(c_1)f(v_1)+\tau(c_2)f(v_2)$.
\end{definition}

\begin{corollary1.5} Let $V$ be a vector space of dimension $\aleph_0$ over a countable field $\mathbb{K}$. Then:
 $$\mathrm{Aut}(\mathrm{Aut}(V)) \cong \mathrm{Aut}(\mathbb{K}^\times) \rtimes \mathrm{Aut}(\mathbb{G}_V).$$
\end{corollary1.5}

\begin{proof}
    First of all, notice that $V$ has the small index property (see \cite[Fact 4.2.11]{hodges}).
    Let $H = \mathrm{Aut}_{\mathrm{SL}}(V)$ be the group of all semilinear automorphisms of $V$. By the Fundamental Theorem of Projective Geometry (see \cite[6.8.3]{ftpgbook}) there exists a homomorphism from $\mathrm{Aut}(\mathbb{G}_V)$ into $H$, $g \mapsto \hat{g} \in H$, satisfying all the hypotheses of Theorem~\ref{genevansth}. Therefore, we can apply Theorem \ref{genevansth} and so we have: 
    $$\mathrm{Aut}(\mathrm{Aut}(V))\cong \mathrm{ker}(\pi) \rtimes \mathrm{Aut}(\mathbb{G}_V).$$

 \noindent
 Now we examine the kernel of $\pi$. Since $\psi$ is an isomorphism, $\mathrm{ker}(\pi)\cong\mathrm{ker}(\xi)$, where $\xi$ is defined as in the proof of \ref{canonicalmaplemma}. We shall show that $\ker(\xi)\cong \mathrm{Aut}(\mathbb{K}^\times)$.

    \noindent First of all, recall that the domain of $\mathcal{E}^{\mathrm{ex}}_V(1)$ is made of triples $(K, p, K')$ where $p: K \cong K'$ and $K$ is a subspace of dimension $0$ or $1$. 
Enumerate the subspaces of dimension $1$ of $V$ as $(K_i : i < \omega)$ and, for every $i < \omega$, choose an element $e_i \in K_i \setminus \{0_V\}$. Then any subspace of dimension $1$ has the form $\{ae_i : a \in \mathbb{K} \} =: \mathbb{K}e_i$ for some $i < \omega$, and so any triple $(K, p, K')$ with $p: K \cong K'$ has the form $a e_i \mapsto \lambda_p a e_j$ for some $i, j < \omega$, where $\lambda_p \in \mathbb{K}^\times$. Thus, for the rest of the proof, we write arbitrary elements of $\mathcal{E}^{\mathrm{ex}}_V(1)$ as $(\mathbb{K}e_i, \lambda, \mathbb{K}e_j)$.

\noindent Let $f \in \mathrm{ker}(\xi)$, then for every $i < \omega$, $f$ sends each triple of the form $(\mathbb{K}e_i, \lambda, \mathbb{K}e_i)$ into a triple of the form $(\mathbb{K}e_i, \lambda', \mathbb{K}e_i)$; so, letting $\lambda' = f_i(\lambda)$, for each $i < \omega$, we have a permutation $f_i$ of $\mathbb{K}^\times$. We will show that:
 \begin{enumerate}[(i)]
 \item for every $i < \omega$, $f_i \in \mathrm{Aut}(\mathbb{K}^\times)$;
 \item for every $i, j < \omega$, $f_i = f_j$.
 \end{enumerate}

\noindent Let $\alpha$ be a map from $\mathrm{ker}(\xi)$ into $\mathrm{Aut}(\mathbb{K}^\times)$ defined by $\alpha(f)=f_i$ for a certain $i \in \omega$. Using (i) and (ii), it is easy to see that the map is well defined and an isomorphism.

\noindent Concerning (i), let $i < \omega$ and $\lambda, \mu \in \mathbb{K}^\times$, then: 
$$\mathbb{K}e_i \xrightarrow{\lambda} \mathbb{K}e_i \xrightarrow{\mu} \mathbb{K}e_i = \mathbb{K}e_i \xrightarrow{\mu\lambda} \mathbb{K}e_i$$
is mapped to:
$$\mathbb{K}e_i \xrightarrow{f_i(\lambda)} \mathbb{K}e_i \xrightarrow{f_i(\mu)} \mathbb{K}e_i = \mathbb{K}e_i \xrightarrow{f_i(\mu\lambda)} \mathbb{K}e_i$$
and so, as $f \in \mathrm{Aut}(\mathcal{E}^{\mathrm{ex}}_V(1))$, we have that:
$$f_i(\mu\lambda) = f_i(\lambda\mu) = f_i(\lambda)f_i(\mu).$$
Similarly, for $i < \omega$ and $\lambda \in \mathbb{K}^\times$ we have
$$\mathbb{K}e_i \xrightarrow{\lambda} \mathbb{K}e_i \xrightarrow{\lambda^{-1}} \mathbb{K}e_i = \mathbb{K} e_i \xrightarrow{1} \mathbb{K} e_i$$
is mapped to:
$$\mathbb{K}e_i \xrightarrow{f_i(\lambda)} \mathbb{K}e_i \xrightarrow{f_i(\lambda^{-1})} \mathbb{K}e_i = \mathbb{K} e_i \xrightarrow{1} \mathbb{K} e_i$$
and so, as $f \in \mathrm{Aut}(\mathcal{E}^{\mathrm{ex}}_V(1))$, we have that:
$$f_i(\lambda^{-1} \lambda) = 1 = f_i(\lambda \lambda^{-1}).$$
Concerning (ii), notice that for $\lambda \in \mathbb{K}^\times$ we obviously have that the following holds:
$$\mathbb{K}e_i \xrightarrow{\lambda} \mathbb{K}e_i \xrightarrow{1} \mathbb{K}e_j \xrightarrow{\lambda^{-1}} \mathbb{K}e_j \xrightarrow{1} \mathbb{K} e_i = \mathbb{K} e_i \xrightarrow{1} \mathbb{K} e_i,$$
and so:
$$\mathbb{K}e_i \xrightarrow{f_i(\lambda)} \mathbb{K}e_i \xrightarrow{1} \mathbb{K}e_j \xrightarrow{f_j(\lambda)^{-1}} \mathbb{K}e_j \xrightarrow{1} \mathbb{K} e_i = \mathbb{K} e_i \xrightarrow{1} \mathbb{K} e_i,$$
from which it follows that $f_i(\lambda) f_j(\lambda)^{-1} = 1$ and so $f_i(\lambda) = f_j(\lambda)$, as desired. This concludes the proof of (i), (ii), and thus the proof.
\end{proof}

\appendix
\section{The topological and algebraic structure of $\mathrm{Aut}(M)$}
In this short appendix, we examine the algebraic and topological structure of $\mathrm{Aut}(M)$ in more detail. The first two propositions provide a characterization of the subgroup and normal subgroup relations between generalized pointwise stabilizers (in the sense of \ref{GS_def}). These results are not needed for our main arguments; we have nonetheless included them here, as we believe they are of independent interest.
\begin{proposition}\label{pre_normality_prop} Let $M$ be a saturated structure. Let $H_1, H_2 \in \mathcal{GS}(M)$. The following conditions are equivalent:
 \begin{enumerate}[(1)]
 \item $H_1 \leq H_2$;
 \item there are $K_1, K_2 \in \mathbf{A}(M)$ and $L_i \leq \mathrm{Aut}_M(K_i)$ such that:
 \begin{enumerate}[(a)]
 \item $H_i = G_{(K_i, L_i)}$;
 \item $K_2 \subseteq K_1$;
 \item for every $f \in L_1$, $f \restriction K_2 \in L_2$. \end{enumerate}
\end{enumerate}
\end{proposition}

\begin{proof}
Concerning (2) implies (1), suppose not, and let $h \in H_1 \setminus H_2$. Since $K_2 \subseteq K_1$, $h \restriction K_1 \in L_1$ but $(h\restriction K_1) \restriction K_2=h\restriction K_2 \notin L_2$, a contradiction.

\noindent Now we show that (1) implies (2).
 $K_2 \subseteq K_1$ follows from Proposition \ref{normality}.
Let $f \in L_1$ and  $\hat{f}\in \mathrm{Aut}(M)$ his expansion. Suppose $f \restriction K_2 \notin L_2$, then $\hat{f} \notin H_2$. Therefore $\hat{f} \in H_1 \setminus H_2$, a contradiction.
\end{proof}

\begin{proposition}\label{normality_prop} Let $M$ be a saturated structure. Let $H_1, H_2 \in \mathcal{GS}(M)$. The following conditions are equivalent:
 \begin{enumerate}[(1)]
 \item $H_1 \trianglelefteq H_2$;
 \item there is $K \in \mathbf{A}(M)$ and $L_1 \trianglelefteq L_2 \leq \mathrm{Aut}_M(K)$ such that $H_i = G_{(K, L_i)}$ for $i =1,2$.
\end{enumerate}
\end{proposition}

 \begin{proof} Follows from Proposition \ref{normality} and the proof of \cite[2.10]{Paolini_BLMS} with minor modifications.
\end{proof}

The next proposition comes from \cite[2.13]{Paolini_BLMS} and holds also in this framework.
\begin{proposition}\label{char_Lpoint_stab} Let $M$ be a saturated structure. Let $L$ be a group and $H \in \mathcal{PS}(M)$. The following conditions are equivalent:
 \begin{enumerate}[(1)]
 \item $H = G_{(K)}$ and $\mathrm{Aut}_M(K) \cong L$;
 \item there is $H' \in \mathcal{GS}(M)$ such that $H \trianglelefteq H'$, $H'$ is maximal under this condition and $H'/H \cong L$.
\end{enumerate}
\end{proposition}

The following proposition gives us a characterization of the $G_\delta$ subgroups in this framework.

\begin{proposition}
Let $M$ be a saturated structure with the Lascar Property. Every $G_\delta$ subgroup of $G$ is of the form $G_{(K,L)}$  with $K$ an algebraically closed subset of $M$ and $L$ a subgroup of ${\mathrm{Aut}}_M(K)$. 
\end{proposition}

\begin{remark}
    We have not defined $G_{(K,L)}$ and $\mathrm{Aut}_M(K)$ for algebraically closed sets, but they are the natural generalization of Definitions \ref{algclosednotation} and \ref{GS_def}.
\end{remark}

\begin{proof}
Let $G_{(K,L)}$ for a certain $K$ and $L$ as above. Since the class of $G_\delta$ subgroups is upward closed, it suffices to observe that $G_{(K)}=G_{(K,\{ \mathrm{id}_{{K}}\})} \leq G_{(K,L)}$ and $G_{(K)}=\bigcap_{a \in K} G_{(a)}$ is $G_\delta$. 

\noindent
Conversely, let $H$ a $G_\delta$ subgroup. Then $H=\bigcap_{n \in \omega} A_n$ with $A_n$ open subgroups. By Lemma \ref{baguettelemma}, $A_n=G_{(K_n,L_n)}$ with $(K_n,L_n) \in \mathbf{EA}(M)$ for every $n \in \omega$. Let the algebraically closed $K=\bigcap_n K_n$ and $L=\{f \restriction K : f \in G  \text{ such that } f\restriction K_n \in L_n\}$. We show that $H=G_{(K,L)}$. Let $f \in H$. Since $f \in G_{(K_n,L_n)}$, $f\restriction K_n \in L_n$ and $f(K)=f(\bigcap_n K_n)=\bigcap_n f(K_n)=\bigcap_n K_n=K$, that is $f \in G_{(K,L)}$. For the other containment, suppose not and let $f \in G_{(K,L)} \setminus H$. Then $f \notin G_{(K_n,L_n)}$ for some $n$. So $f\restriction K_n \notin L_n$, that is, $f \notin L$, a contradiction.
\end{proof}

Under our assumptions, in the $\omega$-categorical case, $\mathcal{GS}(M)$ coincides with the collection of open subgroups of $G = \mathrm{Aut}(M)$ (see \cite{Paolini_BLMS}), while in the strongly minimal case, since every algebraically closed subset of a finitely generated algebraically closed set is itself finitely generated, $\mathcal{GS}(M)$ coincides with the class of $G_\delta$ subgroups. In general, however, $\mathcal{GS}(M)$ lies between these two classes.


\begin{thebibliography}{10}

\bibitem{evan1995}
D. M. Evans and E. Hrushovski.
\newblock{\em The automorphism group of the combinatorial geometry of an algebraically closed field}.
\newblock J. Lond. Math. Soc. (2) {\textbf 52} (1995), 209-225.

\bibitem{Evans_ACF}
D. M. Evans and D. Lascar.
\newblock{\em The automorphism group of the complex numbers is complete}.
\newblock In: Evans DM, editor. Model Theory of Groups and Automorphism Groups.
\newblock London Math. Soc. Lecture Note Ser. (1997), 115-125.

\bibitem{Evans_ind}
D. M. Evans, Z. Ghadernezhad and K. Tent
\newblock{\em Simplicity of the automorphism groups of some Hrushovski constructions}.
\newblock Ann. Pure Appl. Logic {\textbf 167} (2016), n.1, 22-48.

\bibitem{hodges}
W. Hodges.
\newblock {\em Model theory}.
\newblock Cambridge University Press, Cambridge, 1993.

\bibitem{marker2}
D. Marker \newblock {\em Model theory: an introduction}.
\newblock Springer {Nature, GTM} {\textbf 217}, (2002)

\bibitem{marker}
D. Marker \newblock {\em Model theory of fields}.
\newblock {In: D. Marker and A. Pillay, Model Theory of Fields.}
\newblock Cambridge University Press; (2017) p. 38–113.

\bibitem{coarse}
A. Nies, P. Schlicht, and K. Tent
\newblock {\em Coarse groups, and the isomorphism problem for oligomorphic groups}.
\newblock J. Math. Log. {\textbf 22} (2022), no. 01, 2150029.

\bibitem{coarse_PE}
A. Nies and G. Paolini.
\newblock {\em Oligomorphic groups, their automorphism groups, and the complexity of their isomorphism}.
\newblock Beijing J. Pure Appl. Math (in press). Preprint, available on ArXiv at: \url{https://arxiv.org/abs/2410.02248}. 

\bibitem{Paolini_BLMS}
G. Paolini.
\newblock {\em The isomorphism problem for oligomorphic groups with weak elimination of imaginaries}.
\newblock Bull. Lond. Math. Soc. {\textbf 56} (2024), no. 08, 2597-2614.

\bibitem{Paolini_Shelah}
G. Paolini and S. Shelah.
\newblock {\em Reconstructing structures with the small index property up to bi-definability}.
\newblock Fund. Math. {\textbf 247} (2019), 25-35.

\bibitem{poizat}
B. Poizat.
\newblock {\em A course in model theory}.
\newblock Springer-Verlag, New York, 2000.

\bibitem{rothmaler}
P. Rothmaler
\newblock{\em Introduction to model theory}.
\newblock (1st ed.) CRC Press (2000)

\bibitem{ftpgbook}
Ernest Shult.
\newblock{\em Points and Lines: Characterizing the Classical Geometries}.
\newblock Springer Berlin, Heidelberg (2011).

\bibitem{tentziegler}
K. Tent and M. {Ziegler}.
\newblock{\em A course in model theory}.
\newblock Cambridge University Press (2012).

\bibitem{stationary_indep}
K. Tent and M. {Ziegler}.
\newblock{\em On the isometry group of the Urysohn space}.
\newblock J. Lond. Math. Soc. (2) {\textbf 87} (2013), no. 01, 289–303.

\end{thebibliography}
\end{document}